\documentclass [12 pt]{amsart}
\usepackage{amssymb,latexsym,graphicx}
\usepackage{xypic}
\theoremstyle{plain}
\newtheorem{Thm}{Theorem}[section]

\newtheorem{Lem}{Lemma}[section]
\newtheorem{Prop}{Proposition}[section]
\newtheorem{Ex}{Example}[section]
\numberwithin{equation}{section}
\theoremstyle{definition}

\theoremstyle{remark}

\begin{document}

\title[chain cycle]{Partition dimension and strong metric dimension of chain cycle}

\author{Talmeez Ur Rehman $^{[1]}$ and Naila Mehreen $^{[2]}$} 
\date{June 30, 2020}

\address[1]{School of Natural Sciences,
	National University of Sciences and Technology,
	H-12 Islamabad, Pakistan.}

\email{talmeezurehman@gmail.com}

\address[2]{School of Natural Sciences,
	National University of Sciences and Technology,
	H-12 Islamabad, Pakistan.}
\email{nailamehreen@gmail.com}

\keywords{Partition dimension, strong metric dimension, chain cycle}

\subjclass[2000]{05C12}

\begin{abstract}
Let $G$ be a connected graph with vertex set $V(G)$ and edge set $E(G)$. For an ordered $k$-partition $\Pi=\{Q_1,\ldots,Q_k\}$ of $V(G)$, the representation of a vertex $v \in V(G)$ with respect to $\Pi$ is the $k$-vectors $r(v|\Pi)=(d(v,Q_1),\ldots,d(v,Q_k))$, where $d(v,Q_i)$ is the distance between $v$ and $Q_i$.  The partition $\Pi$ is a resolving partition if $r(u|\Pi)\neq r(v|\Pi)$, for each pair of distinct vertices $u,v \in V(G)$. The minimum $k$ for which there is a resolving $k$-partition of $V(G)$ is the partition dimension of $G$. A vertex $w\in V(G)$ strongly resolves two distinct vertices $u,v \in V(G)$ if $u$ belongs to a shortest $v-w$ path or $v$  belongs to a shortest $u-w$ path. An ordered set $W=\{w_{1},\ldots, w_{t}\}\subseteq V(G)$ is a strong resolving set for $G$ if for every two distinct vertices $u$ and $v$ of $G$ there exists a vertex $w\in W$ which strongly resolves $u$ and $v$.
A strong metric basis of $G$ is a strong resolving set of minimal cardinality. The cardinality of a strong metric basis is called strong metric dimension of $G$. In this paper, we determine the partition dimension and strong metric dimension of a chain cycle constructed by even cycles and a chain cycle constructed by odd cycles.
\end{abstract}

\maketitle

\section{Introduction}
\label{IMV01.S.1}
Let $G$ be a finite, simple and connected graph with vertex set $V(G)$ and edge set $E(G)$. 
The distance between two vertices $u$ and $v$ of $G$ is the length of the shortest path from $u$ to $v$ in $G$ and is denoted by $d(u,v)$. Two distinct vertices $u$ and $v$ are called adjacent if there is an edge between them and denoted by $uv$. The degree of a vertex $u$ is the number of vertices adjacent to it and is denoted by $d_{G}(v)$ or simply $d(v)$. The set of neighborhood of a vertex $u\in V(G)$, denoted by $N(u)$, is the set of all vertices of $G$ that are adjacent to $u$.The diameter of a graph $G$, denoted by $d(G)$, is defined as $d(G)=max\{d(u,v)\mid u,v\in V(G)\}$ . A cycle of lenght $n$ is denoted by $C_{n}$. 

A vertex $w\in V(G)$ resolves two vertices $u$ and $v$ of $G$ if $d(w,u)\neq d(w,v)$.
An ordered set $W=\{w_{1},\ldots, w_{t}\}\subseteq V(G)$ is a resolving set for $G$ if for every two distinct vertices $u$ and $v$ of $G$ there exists a vertex $w\in W$ which resolves $u$ and $v$. The representation of a vertex $u\in V(G)$ with respect to $W$, denoted by $r_{G}(u|W)$, is defined by the $t$-vectors $r_{G}(u|W)=(d(u,w_{1}),d(u,w_{2}),\dots,d(u,w_{t}))$. The metric basis of $G$ is a resolving set of minimal cardinality. The cardinality of the metric basis is called metric dimension of $G$ and is denoted by
$\emph{dim}(G)$. The metric dimension of graphs was introduced independently by Harary and Melter in \cite{H1}. For more detail see \cite{A, C2, H1, I1, K2, T1}.

Later on the concept of partition dimension was given by Chartrand et al. \cite{C1} in $2000$. Given an ordered partition $\Pi=\{Q_1,\ldots,Q_t\}$ of the vertices of $G$, the partition representation of a vertex $u \in V(G)$ with respect to $\Pi$ is the vector $r(u|\Pi)=(d(u,Q_1),\ldots,d(u,Q_t)),$
where $d(u,Q_j)=\min\{d(u, q)\mid q\in Q_j\}$, for each $j=1,2,\ldots,t$.
The partition $\Pi$ is a resolving partition of $G$ if for every pair of distinct vertices $u,v \in V(G)$, $r(u|\Pi)\neq r(v|\Pi)$.
The partition dimension of $G$ is the cardinality of a minimum resolving partition of $G$ and is denoted by $\emph{pd}(G)$. See \cite{C1, J1, M2, R2, T2} for more results. 

Seb{\H{o}} and Tannier \cite{S1}, in $2004$, gave more strict version of metric dimension of a graph called the strong metric dimension of a graph. A vertex $w\in V(G)$ strongly resolves two distinct vertices $u,v \in V(G)$ if $u$ belongs to a shortest $v-w$ path or $v$  belongs to a shortest $u-w$ path. An ordered set $W=\{w_{1},\ldots, w_{t}\}\subseteq V(G)$ is a strong resolving set for $G$ if for every two distinct vertices $u$ and $v$ of $G$ there exists a vertex $w\in W$ which strongly resolves $u$ and $v$.
A strong metric basis of $G$ is a strong resolving set of minimal cardinality. The cardinality of a strong metric basis is called strong metric dimension of $G$ and is denoted by
$\emph{sdim}(G)$. For more detail, see \cite{K1,K3,O1,R1}.

A set $S$ of vertices of $G$ is a vertex cover of $G$ if every edge of $G$ is incident with at least one vertex of $S$. The vertex cover number of $G$, denoted by $\alpha(G)$, is the smallest cardinality of a vertex cover of $G$. The largest cardinality of a set of vertices of $G$, no two of which are adjacent, is called the independence number of $G$ and is denoted by $\beta(G)$. Since for any graph $G$ of order $n$ the complement of an independent set $S\subseteq V(G)$ is a vertex cover of $G$ and therefore $\alpha(G)+\beta(G)=n$.

A vertex $u\in V(G)$ is maximally distant from $v\in V(G)$, denoted by $u$MD$v$, if for every vertex $w$ in the neighborhood of $u$, $ d_{G}(v,w)\leq
d_{G}(u,v)$. If $u$ is maximally distant from $v$ and $v$ is
maximally distant from $u$, then we say that $u$ and $v$ are
mutually maximally distant and we denote it as $u$MMD$v$. The strong resolving graph of $G$ is a graph $G_{SR}$ whose vertex set is $V(G)$ and two vertices $u,v\in V(G)$ are adjacent in
$G_{SR}$ if and only if $u$MMD$v$. Oellermann and Peters-Fransen \cite{O1} showed that finding
the strong metric dimension of a connected graph $G$ is equivalent
to finding the vertex cover number of $G_{SR}$.
\begin{Thm}[Oellermann and Peters-Fransen \cite{O1}]\label{T3}
	For any connected graph G, $sdim(G)=\alpha(G_{SR})$.
\end{Thm}

Let $\{G_i\}_{i=1}^m$ be a set of finite pairwise disjoint simple connected graphs. The chain graph 
\begin{equation*}
\mathcal{C}(G_1, G_2, \ldots ,G_m) = \mathcal{C}(G_1, G_2, . . . ,G_m; x_1, w_1, x_2, w_2, \ldots , x_m, w_m) \nonumber
\end{equation*}
of $\{G_i\}_{i=1}^m$  with respect to the vertices $\{x_i,w_i \in V(G_i) \mid i=1,2,\ldots,m\}$ is the graph obtained from the graphs $G_1, \ldots , G_m$ by identifying the vertex $w_i$ and the vertex $x_{i+1}$, as shown in Figure \ref{cycle42}, for all $i \in \{1, 2, \ldots , m-1\}$. For more results and detail about chain graph, see \cite{M1, N2}.
\begin{figure}[h!]\label{cycle42}
	\centering
	\includegraphics[width=12cm]{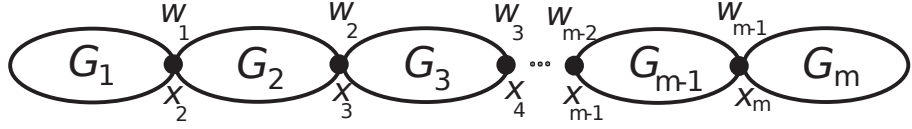}
	\caption{\small{A chain graph}}
\end{figure}
Let $\{C_{n_i}\}_{i=1}^m$ be a set of finite pairwise disjoint simple cycles. Let $V(C_{n_i})=\{v^{i}_{j}\mid j=1,2,\ldots,n_{i}\}$, where $i\in\{1,2,\dots,m\}$. Assume that $n_i$ is even for each $ i=1,2,\dots,m $. We consider a chain cycle of $\{C_{n_i}\}_{i=1}^m$ given by 
\begin{align}\nonumber
\begin{split}
&\mathcal{C}(C_{n_1}, C_{n_2}, \ldots ,C_{n_m})
\\
&= \mathcal{C}\left( C_{n_1},C_{n_2},\dots,C_{n_m};v^1_1,v_1^2,v_{\frac{n_{1}}{2}+1}^1, v_1^3, v_{\frac{n_{2}}{2}+1}^2, \ldots ,v_1^m,
v_{\frac{n_{m-1}}{2}+1}^{m-1},
v^m_{\frac{n_{m}}{2}+1}\right)
\end{split}
\end{align} 
with respect to the vertices $\{v^{i}_{\frac{n_{i}}{2}+1}, v^{i+1}_{1} \mid i=1,2,\dots,m-1 \}$. A chain cycle of $ \{C_8, C_{10}, C_8\} $ with respect to vertices $ \{v_5^1,v_1^2,v^2_{6} v_1^3\} $ is shown in Figure \ref{cycle4}.

Now, assume that $n_i$ is odd for each $i=1,2,\dots,m$. We consider a chain cycle of $\{C_{n_i}\}_{i=1}^m$ given by 
\begin{align}\nonumber
\begin{split}
&\mathcal{C}(C_{n_1}, C_{n_2}, \ldots ,C_{n_m}) 
\\
&= \mathcal{C}\left( C_{n_1}, C_{n_2}, \ldots ,C_{n_m};v^1_1, v_1^2, v_{\frac{n_{1}+1}{2}+1}^1, v_1^3,v_{\frac{n_{2}+1}{2}+1}^2,  \ldots ,v_1^m, v_{\frac{n_{m-1}+1}{2}+1}^{m-1},v^m_{\frac{n_{m}+1}{2}+1}\right)
\end{split}
\end{align} 
with respect to the vertices $\{v^{i}_{\frac{n_{i}+1}{2}+1}, v^{i+1}_{1} \mid i=1,2,\dots,m-1\}$. A chain cycle of $ \{C_5, C_{7}, C_5\} $ with respect to vertices $ \{v_4^1,v_1^2,v_5^2, v_1^3\} $ is shown in Figure \ref{cycle41}.

Through out the paper, we denote the vertex set and the edge set of chain cycle by $V(\mathcal{C})$ and $E(\mathcal{C})$ instead of $V(\mathcal{C}(C_{n_1}, C_{n_2}, \ldots ,C_{n_m}))$ and $E(\mathcal{C}(C_{n_1}, C_{n_2}, \ldots ,C_{n_m}))$, respectively.
\section{Partition dimension of chain cycles}
In this section, we find the partition dimension of chain cycle constructed by even cycles and chain cycle constructed by odd cycles. Following two results are important tools for proving our results.  
\begin{Thm}[Chartrand et al. \cite{C1}]\label{th1}
	If $G$ is a nontrivial connected graph, then $pd(G)\leq dim(G)+1$.
\end{Thm}
\begin{Prop}[Chartrand et al. \cite{C1}]\label{pro1}
	Let $G$ be a connected graph of order $n\geq2$. Then $pd(G)=2$ if and only if $G\cong P_n$.
\end{Prop}
In the following theorem, we compute the partition dimension of chain cycle constructed by even cycles.
\begin{Thm}\label{TT1}
	The partition dimension of chain cycle $\mathcal{C}(C_{n_1}, C_{n_2}, \ldots ,C_{n_m})= \mathcal{C}\Big( C_{n_1},$
	$ C_{n_2}, \ldots
	 ,C_{n_m};v^1_1,v_1^2,v_{\frac{n_{1}}{2}+1}^1, v_1^3, v_{\frac{n_{2}}{2}+1}^2, \ldots ,v_1^m, v_{\frac{n_{m-1}}{2}+1}^{m-1},v_{\frac{n_{m}}{2}+1}^{m}\Big)$ is $3$, where $n_i$ is even for each $i=1,2,\dots,m$.
\end{Thm}
\begin{proof}
	
	Let $\Pi=\{Q_{1},Q_{2},Q_{3}\}$, where $Q_{1}=\{v^{1}_{1},\ldots,v^{1}_{\frac{n_1}{2}-1},v^{1}_{\frac{n_1}{2}+2},\ldots,v^1_{n_1}\}$, $Q_{2}=\{v^{2}_{\frac{n_{2}}{2}+3},v^{2}_{\frac{n_{2}}{2}+4},\ldots,$ $v^{2}_{n_{2}},v^{3}_{\frac{n_{3}}{2}+3},v^{3}_{\frac{n_{3}}{2}+4}
	,\ldots,$
	$v^{3}_{n_{3}},\ldots, v^{m-1}_{\frac{n_{m-1}}{2}+3},v^{m-1}_{\frac{n_{m-1}}{2}+4},\ldots,
	v^{m-1}_{n_{m-1}}\}\cup\{v^{m}_{n_{m}}\}$ and  $Q_{3}=V(\mathcal{C})\setminus \{Q_{1} \cup Q_{2}\}$ be a partition of $V(\mathcal{C})$.
	We show that $\Pi$ is a resolving partition of $V(\mathcal{C})$ with minimum cardinality.
	The representation of each vertex of $V(\mathcal{C})$ with respect to $\Pi$ is given as:
	\begin{eqnarray*}
		r(v^{1}_{\frac{n_1}{2}}|\Pi)=(1,2,0), & r(v^{1}_{\frac{n_1}{2}+1}|\Pi)=(1,1,0),& r(v_{n_m}^m|\Pi)=\left( \sum\limits_{k=2}^{m} \frac{n_{k}}{2}+2,0,1\right) .
	\end{eqnarray*}
	\begin{equation*}
	r(v^{1}_{j}|\Pi)=\left\{
	\begin{array}{l l}
	(0,{n_1}-j-1, {n_1}-j-3)    & \mbox{if $1\leq j\leq \frac{n_{1}}{2}-1$} \\
	(0,j-\frac{n_1}{2},j-\frac{n_1}{2}-1)     & \mbox{if $\frac{n_{1}}{2} +2\leq j\leq n_{1}$},
	\end{array}\right.
	\end{equation*}
	
	\begin{equation*}
	r(v^{i}_{j}|\Pi)=\left\{
	\begin{array}{l l}
	(j,j,0)                                                              & \mbox{if $1\leq j\leq \lceil \frac{n_{2}}{4}\rceil$} \\
	(j,\frac{n_{2}}{2} -j+2,0)                                           & \mbox{if $\lceil \frac{n_{2}}{4}\rceil+1 \leq j\leq \frac{n_{2}}{2}$},
	\end{array}\right.
	\end{equation*}
	\begin{equation*}
	r(v^{i}_{j}|\Pi)=\left\{
	\begin{array}{l l}
	\left( \sum\limits_{k=3}^{m} \frac{n_{k}}{2}+j,j,0 \right)                        & \mbox{if $1\leq j\leq \lceil \frac{n_{i}}{4}\rceil, 3\leq i \leq m$} \\
	\left(\sum\limits_{k=3}^{m}\frac{n_{k}}{2}+j,\frac{n_{i}}{2}-j+2,0 \right)       & \mbox{if $\lceil \frac{n_{i}}{4}\rceil+1 \leq j\leq \frac{n_{i}}{2}, 3\leq i \leq m$},
	\end{array}\right.
	\end{equation*}
	
	\begin{equation*}
	r(v^{i}_{\frac{n_{i}}{2}+j}|\Pi)=\left\{
	\begin{array}{l l}
	\left(\sum\limits_{k=2}^{m-1} \frac{n_{k}}{2},1,0 \right)                     & \mbox{if $2\leq i\leq m-1, j=2$} \\
	\left(\sum\limits_{k=2}^{m-1} \frac{n_{k}}{2}+n_{m}-j+2,n_{i}-j,0 \right)     & \mbox{if $i=m, \frac{n_{i}}{2}+1\leq j\leq n_i -1$},
	\end{array}\right.
	\end{equation*}
	
	\begin{equation*}
	r(v^{i}_{j}|\Pi)=\left\{
	\begin{array}{l l}
	(n_{i}+2-j,0,j-\frac{n_{i}}{2}-2)                                            & \mbox{if $\frac{n_{2}}{2}+3\leq j\leq \lceil \frac{3n_{2}}{4}\rceil+1$} \\
	(n_{i}+2-j,0,n_{i}+1-j)                                                      & \mbox{if $\lceil \frac{3n_{2}}{4}\rceil+2 \leq j\leq n_{2}$},
	\end{array}\right.
	\end{equation*}
	\begin{equation*}
	r(v^{i}_{j}|\Pi)=\left\{
	\begin{array}{l l}
	\left(\sum\limits_{k=3}^{m} \frac{n_{k}}{2}+n_{i}-j,0,j-\frac{n_{i}}{2} -2 \right)       & \mbox{if $\frac{n_{i}}{2}+3\leq j\leq \lceil \frac{3n_{i}}{4}\rceil+1, 3\leq i \leq m $} \\
	\left(\sum\limits_{k=3}^{m} \frac{n_{k}}{2}+n_{i}-j,0,n_{i}+1-j \right)                  & \mbox{if $\lceil \frac{3n_{i}}{4}\rceil+2 \leq j\leq n_{i}, 3\leq i \leq m$}.
	\end{array}\right.
	\end{equation*}
It is easily seen that the representation of each vertex with respect to $\Pi$ is distinct.
This shows that $\Pi$ is a resolving partition of $\mathcal{C} (C_{n_1}, C_{n_2}, \ldots ,C_{n_m})$. Thus $pd(\mathcal{C}(C_{n_1}, C_{n_2}, \ldots ,C_{n_m}))\leq 3$.
	
On the other hand, by Proposition \ref{pro1} it follows that $pd(\mathcal{C}( C_{n_1}, C_{n_2}, \ldots ,C_{n_m}))\geq3$. Hence $\emph{pd}(\mathcal{C}( C_{n_1}, C_{n_2}, \ldots ,C_{n_m}))=3$.
\end{proof}
In the following example, we find the partition dimension of a chain cycle constructed by $C_{8}$, $C_{10}$ and $C_{8}$.
\begin{Ex}\label{e1}
Let $m=3$ and $n_{1}=8$, $n_{2}=10$ and $n_{3}=8$. The chain cycle constructed by $C_8$, $C_{10}$ and $C_8$ with respect to the vertices $ \{v_5^1,v_1^2,v^2_{6}, v_1^3\}$ is denoted by $\mathcal{C}(C_{8}, C_{10},C_{8})=\mathcal{C}( C_{8}, C_{10}, C_{8}; v_1^1, v^{2}_{1}, v^{1}_{5}, v^{3}_{1},$ $ v^{2}_{6}, v_5^3) $ and is given in Figure \ref{cycle4}.
	\begin{figure}[h!]
		\centering
		\includegraphics[width=10cm]{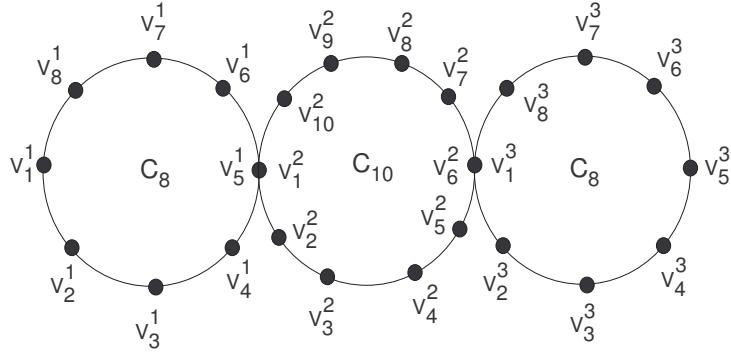}
		\caption{\small{Chain cycle of $C_8$, $C_{10}$ and $C_8$}}\label{cycle4}
	\end{figure}
	
	Using Theorem \ref{TT1}, we construct a resolving partition of $\mathcal{C}(C_{8}, C_{10},C_{8})$ as  $\Pi=\{Q_{1},Q_{2},Q_{3}\}$, where $Q_{1}=\{v^{1}_{1},v^{1}_{2}, v^{1}_{3}, v^{1}_{6}, v^{1}_{7}, v^{1}_{8}\}$, $Q_{2}=\{v^{2}_{8},v^{2}_{9}, v^{2}_{10}, v^{3}_{8}\}$ and $Q_{3}=\{v^{1}_{4},v^{1}_{5}, v^{2}_{2}, v_{3}^{2}, v_{4}^{2}, v_{5}^{2}, v_{6}^{2}, v_{7}^{2}, v_{2}^{3}, v_3^3,$ 
	$ v_4^3, v_4^3, v_5^3, v_6^3, v_7^3\}$. Again by Theorem \ref{TT1}, we note that each vertex of $\mathcal{C}(C_{8}, C_{10},C_{8})$ has distinct representation with respect to $\Pi$, as shown in Table $1$. Hence $pd(\mathcal{C}(C_{8}, C_{10},C_{8}))$ 
	$=3$.
	
	\begin{table}
		\centering
		
		\begin{tabular}{|c|c|c|}
			\hline
			$r(v_1^1|\Pi)=(0,5,3) $ & $r(v_2^2|\Pi)=(2,2,0)$ & $r(v_{10}^2|\Pi)=(2,0,1)$ \\ \hline
			$r(v_2^1|\Pi)=(0,4,2) $ & $r(v_3^2|\Pi)=(3,3,0)$ & $r(v_2^3|\Pi)=(7,2,0)$ \\ \hline
			$r(v_3^1|\Pi)=(0,3,1) $ & $r(v_4^2|\Pi)=(4,3,0)$ & $r(v_3^3|\Pi)=(8,3,0)$ \\ \hline
			$r(v_4^1|\Pi)=(1,2,0) $ & $r(v_5^2|\Pi)=(5,2,0)$ & $r(v_4^3|\Pi)=(9,4,0)$ \\ \hline
			$r(v_5^1|\Pi)=(1,1,0) $ & $r(v_6^2|\Pi)=(6,1,0)$ & $r(v_5^3|\Pi)=(10,3,0)$ \\ \hline
			$r(v_6^1|\Pi)=(0,2,1) $ & $r(v_7^2|\Pi)=(5,1,0)$ & $r(v_6^3|\Pi)=(9,2,0)$ \\ \hline
			$r(v_7^1|\Pi)=(0,3,2) $ & $r(v_8^2|\Pi)=(4,0,1)$ & $r(v_7^3|\Pi)=(8,1,0)$ \\ \hline
			$r(v_8^1|\Pi)=(0,4,3) $ & $r(v_9^2|\Pi)=(3,0,2)$ & $r(v_8^3|\Pi)=(7,0,1)$ \\
			\hline
		\end{tabular}\label{SVLP}
		\caption{Representation of $v^{i}_{j}$ with respect to $\Pi$} 
	\end{table}
\end{Ex}
\begin{Thm}\label{TT2}
	The partition dimension of chain cycle $\mathcal{C}(C_{n_1}, C_{n_2},$ 
	$ \ldots ,C_{n_m})= \mathcal{C}\Big(C_{n_1},$
	$ C_{n_2}, \ldots , C_{n_m};v^1_1, v_1^2, v_{\frac{n_{1}+1}{2}+1}^1, v_1^3,v_{\frac{n_{2}+1}{2}+1}^2,  \ldots ,v_1^m, v_{\frac{n_{m-1}+1}{2}+1}^{m-1},v_{\frac{n_{m}+1}{2}+1}^{m}\Big)$ is $3$, where $n_i$ is odd for each $i=1,2,\dots,m$.
\end{Thm}
\begin{proof}
Let $\Pi=\{Q_{1},Q_{2},Q_{3}\}$, where $Q_{1}=\{v^{1}_{1},\ldots,v^{1}_{\lceil\frac{n_1}{2}\rceil-1},v^{1}_{\lceil\frac{n_1}{2}\rceil+2},\ldots,v^1_{n_1}\}$, $Q_{2}=\{v^{2}_{\lceil\frac{n_{2}}{2}\rceil+3},$ $v^{2}_{\lceil\frac{n_{2}}{2}\rceil+4},\ldots,v^{2}_{n_{2}},$ $v^{3}_{\lceil\frac{n_{3}}{2}\rceil+3},v^{3}_{\lceil\frac{n_{3}}{2}\rceil+4},\ldots,v^{3}_{n_{3}}, \ldots, v^{m-1}_{\lceil\frac{n_{m-1}}{2}\rceil+3},v^{m-1}_{\lceil\frac{n_{m-1}}{2}\rceil+4},\ldots,v^{m-1}_{n_{m-1}}\}$
$\cup \{v^{m}_{n_{m}}\}$ and $Q_{3}=V(\mathcal{C})\setminus \{Q_{1} \cup Q_{2}\}$ be a partition of $V(\mathcal{C})$.
We show that $\Pi$ is a resolving partition of $V(\mathcal{C})$ with minimum cardinality.
The representation of each vertex of $V(\mathcal{C})$ with respect to $\Pi$ is given as:
\begin{eqnarray*}
	r(v^{1}_{\lceil \frac{n_1}{2} \rceil}|\Pi)=(1,2,0), & r(v^{1}_{\lceil \frac{n_1}{2} \rceil+1}|\Pi)=(1,1,0),& r(v_{n_m}^m|\Pi)=\left( \sum\limits_{k=2}^{m} \lfloor\frac{n_{k}}{2}\rfloor+2,0,1\right).
\end{eqnarray*}

\begin{equation*}
r(v^{1}_{j}|\Pi)=\left\{
\begin{array}{l l}
(0,{n_1}-\lfloor\frac{n_1}{2}\rfloor, {n_1}-\lfloor\frac{n_1}{2}\rfloor-1) & \mbox{if j=1 } \\
(0,{n_1}-j-1, {n_1}-j-3)      & \mbox{if $2\leq j\leq \lceil\frac{n_{1}}{2}\rceil-1$ } \\
(0,j-\lceil\frac{n_1}{2}\rceil,j-\lceil\frac{n_1}{2}\rceil-1) & \mbox{if $\lceil\frac{n_{1}}{2}\rceil +2\leq j\leq n_{1}$},
\end{array}\right.
\end{equation*}

\begin{equation*}
r(v^{i}_{j}|\Pi)=\left\{
\begin{array}{l l}
(j,j,0)                                                            & \mbox{if $1\leq j\leq \lceil \frac{n_{2}}{4}\rceil+1$} \\
(j,\lceil\frac{n_{2}}{2}\rceil-j+3,0)                              & \mbox{if $\lceil \frac{n_{2}}{4}\rceil+2 \leq j\leq \lceil\frac{n_{2}}{2}\rceil+1$},
\end{array}\right.
\end{equation*}
\begin{equation*}
r(v^{i}_{j}|\Pi)=\left\{
\begin{array}{l l}
\left(\sum\limits_{k=3}^{m} \lfloor\frac{n_{k}}{2}\rfloor+j ,j,0 \right)       & \mbox{if $1\leq j\leq \lceil \frac{n_{i}}{4}\rceil+1, 3\leq i \leq m $} \\
\left(\sum\limits_{k=3}^{m} \lfloor\frac{n_{k}}{2}\rfloor+j,\lceil\frac{n_{i}}{2}\rceil-j+3,0 \right)       & \mbox{if $\lceil \frac{n_{i}}{4}\rceil+2 \leq j\leq \lceil\frac{n_{i}}{2}\rceil+1, 3\leq i \leq m$},
\end{array}\right.
\end{equation*}

\begin{equation*}
r(v^{i}_{\frac{n_{i}}{2}+j}|\Pi)=\left\{
\begin{array}{l l}
\left(\sum\limits_{k=2}^{m-1} \lfloor\frac{n_{k}}{2}\rfloor,1,0 \right)                     & \mbox{if $2\leq i\leq m-1, j=2$} \\
\left(\sum\limits_{k=2}^{m-1} \lfloor\frac{n_{k}}{2}\rfloor+n_{m}-j+2,n_{i}-j,0 \right)     & \mbox{if $i=m,\lceil\frac{n_{i}}{2}\rceil+2\leq j\leq n_i-1$},
\end{array}\right.
\end{equation*}

\begin{equation*}
r(v^{i}_{j}|\Pi)=\left\{
\begin{array}{l l}
(n_{i}+2-j,0,j-\lceil\frac{n_{i}}{2}\rceil -2)                            & \mbox{if $\lceil\frac{n_{2}}{2}\rceil+3\leq j\leq \lceil \frac{3n_{2}}{4}\rceil+1$} \\
(n_{i}+2-j,0,n_{i}+1-j)                                                   & \mbox{if $\lceil \frac{3n_{2}}{4}\rceil+2 \leq j\leq n_{2}$},
\end{array}\right.
\end{equation*}
\begin{equation*}
r(v^{i}_{j}|\Pi)=\left\{
\begin{array}{l l}
\left(\sum\limits_{k=3}^{m} \lceil\frac{n_{k}}{2}\rceil+n_{i}-j,0,j-\lceil\frac{n_{i}}{2}\rceil-2 \right)       & \mbox{if $\lceil\frac{n_{i}}{2}\rceil+3\leq j\leq \lceil \frac{3n_{i}}{4}\rceil+1, 3\leq i \leq m $} \\
\left(\sum\limits_{k=3}^{m} \lceil\frac{n_{k}}{2}\rceil+n_{i}-j,0,n_{i}+1-j \right)       & \mbox{if $\lceil \frac{3n_{i}}{4}\rceil+2 \leq j\leq n_{i}, 3\leq i \leq m$}.
\end{array}\right.
\end{equation*}
It is easily seen that the representation of each vertex with respect to $\Pi$ is distinct.
This shows that $\Pi$ is a resolving partition of $\mathcal{C}(C_{n_1}, C_{n_2}, \ldots ,C_{n_m})$. Thus $\emph{pd}(\mathcal{C}(C_{n_1}, C_{n_2}, \ldots ,C_{n_m}))\leq 3$.

On the other hand, by Proposition \ref{pro1} it follows that $\emph{pd}(\mathcal{C}(C_{n_1}, C_{n_2}, \ldots ,C_{n_m}))\geq3$. Hence $\emph{pd}(\mathcal{C}(C_{n_1}, C_{n_2}, \ldots ,C_{n_m}))=3$.
\end{proof}
In the following example, we compute the partition dimension of the chain cycle constructed by $C_{5}$, $C_{7}$ and $C_{5}$.
\begin{Ex}\label{e2}
Let $m=3$ and $n_{1}=5$, $n_{2}=7$ and $n_{3}=5$. The chain cycle constructed by $C_5$, $C_{7}$ and $C_5$ with respect to vertices $ \{v_4^1,v_1^2,v_5^2, v_1^3\} $ is denoted by $\mathcal{C}(C_{5}, C_{7},C_{5})=\mathcal{C}(C_{5}, C_{7},C_{5};v^1_1,v^{2}_{1},v^{1}_{4},$
$v^{3}_{1},v^{2}_{5},v^3_4) $ and is given in Figure \ref{cycle41}.
\begin{figure}[h!]
	\centering
	\includegraphics[width=10cm]{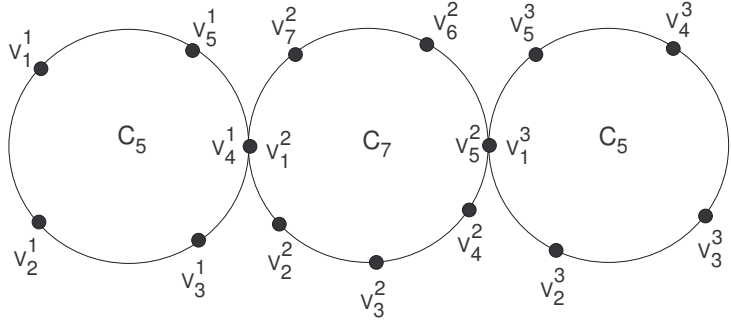}
	\caption{\small{Chain cycle of $C_{5}$, $C_{7}$ and $C_{5}$}}\label{cycle41}
\end{figure}

Using Theorem \ref{TT2}, we construct a resolving partition of $\mathcal{C}(C_{5}, C_{7},C_{5})$ as $\Pi=\{Q_{1},Q_{2},Q_{3}\}$, where $Q_{1}=\{v^{1}_{1}, v^{1}_{2}, v^{1}_{5}\}$, $Q_{2}=\{v^{2}_{7},v^{3}_{5}\}$ and $Q_{3}=\{v^{1}_{3},v^{1}_{4}, v^{2}_{2}, v_{3}^{2}, v_{4}^{2}, v^{2}_{5}, v_{6}^{2},$ $ v_2^3, v_3^3, v_4^3\}$. Again by Theorem \ref{TT2}, we note that each vertex of $\mathcal{C}(C_{5}, C_{7},C_{5})$ has distinct representation with respect to $\Pi$, as shown in Table $2$. Hence $pd(\mathcal{C}(C_{5}, C_{7},C_{5}))$ 
$=3$.

\begin{table}
	\centering
	
	\begin{tabular}{|c|c|}
		\hline
		$r(v_1^1|\Pi)=(0,3,2) $&$r(v_5^2|\Pi)=(4,1,0)$  \\ \hline
		$r(v_2^1|\Pi)=(0,3,1) $&$r(v_6^2|\Pi)=(3,1,0)$  \\ \hline
		$r(v_3^1|\Pi)=(1,2,0) $&$r(v_7^2|\Pi)=(2,0,1)$  \\ \hline
		$r(v_4^1|\Pi)=(1,1,0) $&$r(v_2^3|\Pi)=(5,2,0)$  \\ \hline
		$r(v_5^1|\Pi)=(0,2,1) $&$r(v_3^3|\Pi)=(6,2,0)$  \\ \hline
		$r(v_2^2|\Pi)=(2,2,0) $&$r(v_4^3|\Pi)=(6,1,0)$  \\ \hline
		$r(v_3^2|\Pi)=(3,3,0) $&$r(v_5^3|\Pi)=(5,0,1)$  \\ \hline
		$r(v_4^2|\Pi)=(4,2,0) $&  \\
		\hline
	\end{tabular}\label{SVLP1}
	\caption{Representation of $v^{i}_{j}$ with respect to $\Pi$} 
\end{table}
\end{Ex}
\section{Strong metric dimension of chain cycle}
In this section, we find the strong metric dimension of the chain cycle constructed by even cycles and the chain cycle constructed by odd cycles. 
Let $V_{1}=\{v^2_{1},v^{3}_{1},\dots, v^m_{1}\}$ and $V_2=V(\mathcal{C})\setminus V_{1}$.
Through out the section, we denote the strong resolving graph of a chain cycle $\mathcal{C}(C_{n_1},C_{n_2},\dots,C_{n_m})$ by $\mathcal{C}_{SR}(C_{n_1},C_{n_2},$ 
$\dots,C_{n_m})$. Furthermore, we denote the vertex set and the edge set of the strong resolving graph of  $\mathcal{C}(C_{n_1},C_{n_2},\dots,C_{n_m})$ by $V(\mathcal{C}_{SR})$ and $E(\mathcal{C}_{SR})$, respectively.  

Following two lemmas are easy observations from the structure of a cycle $C_{n}$ and a chain cycle constructed by even cycles as well as a chain cycle constructed by even cycles, respectively.  
\begin{Lem}\label{L9}
	Let $C_n$ be a cycle. Then for two distinct vertices $u_{i},u_{j}\in V(C_{n})$ we have $u_i$MMD$u_j$ if and only if $d(u_i, u_j)=d(C_{n})$.
\end{Lem}
\begin{Lem}\label{L2}
	Let $x\in V_1$ and $y\in V(\mathcal{C})$. Then $x$ and $y$ are not mutually maximally distant.  
\end{Lem}
In the next theorem, we find the mutually maximally distant vertices in chain cycle  $\mathcal{C}(C_{n_1},C_{n_2},$ 
$\dots,C_{n_m})$, with each $n_i$ is even, with respect to the vertices $\{v^{i}_{\frac{n_{i}}{2}+1},v^{i+1}_1 \mid i=1,2,\dots,m-1 \}$.  Here, we denote $U_{1}(C_{n_i})=\{v^{i}_{1},v^{i}_{2},\dots,v^{i}_{\frac{n_{i}}{2}}\}$ and $U_{2}(C_{n_i})=\{v^{i}_{\frac{n_{i}}{2}+1},v^{i}_{\frac{n_{i}}{2}+2},$ $\dots,v^i_{n_i}\}, i\in\{1,2,\dots,m\}$. 
\begin{Thm}\label{T1}	
	Let $v^i_j,v^k_l\in V_2$, where $i,k\in \{1,2,\dots,m\}$, in a chain cycle $\mathcal{C}(C_{n_1},C_{n_2},$
	$\dots,C_{n_m})$ constructed by even cycles with respect to the vertices $\{v^{i}_{\frac{n_{i}}{2}+1},v^{i+1}_1\mid i=1,2,\dots, m-1\}$.
	\begin{description}
		\item[$a$] 
		Let $i=k$. Then $v^i_j$MMD$v^i_l$ if and only if $d(v^i_j,v^i_l)=d(C_{n_i})$. 
		\item[$b$]
		Let $i\neq k$. Then $v^i_j$MMD$v^k_l$ if and only if $d(v^i_j,v^k_l)=d(\mathcal{C}(C_{n_1},C_{n_2},\dots,C_{n_m}))$.
	\end{description}
\end{Thm}
\begin{proof} 
	$(a).$	Let $d(v^i_j,v^i_l)=d(C_{n_i})$. Then from Lemma \ref{L9}, we have $v^i_j$MMD$v^i_l$.
	
	Conversely, let  $v^i_j$MMD$v^i_l$ and $j<l$. On the contrary, assume that  $d(v^i_j,v^i_l)< d(C_{n_i})$. Since $N(v^i_j)=\{v^i_{j-1},v^i_{j+1}\}$ and $N(v^i_l)=\{v^i_{l-1},v^i_{l+1}\}$. Note that either   $v^i_jv^i_{j+1}\dots v^{i}_l$ or $v^i_jv^i_{j-1}\dots v^{i}_1v^i_{n_i}\dots v^i_l$ is a shortest path from $v^i_j$ to $v^i_l$. This shows that $v^i_j$ and $v^i_l$ are not mutually maximally distant which contradicts our supposition that $v^i_j$MMD$v^i_l$.
	
	$(b).$ Let $d(v^i_j,v^k_l)=d(\mathcal{C}(C_{n_1},C_{n_2},\dots,C_{n_m}))$. Then clearly $v^i_j$MMD$v^k_l$. 
	
	Conversely, let $u^i_j$MMD$u^k_l$ and let  $d(v^i_j,v^k_l)<d(\mathcal{C}(C_{n_1},C_{n_2},\dots,C_{n_m}))$. Since $N(v^{i}_{j})=\{v^i_{j-1},$ 
	$v^i_{j+1}\}$ and $N(v^{k}_{l})=\{v^k_{l-1},u^k_{l+1}\}$. If $v^i_j\in U_1(C_{n_i})$ and $v^k_l\in U_1(C_{n_k})$, then $P_1=v^i_jv^i_{j+1}\dots v^{i+1}_1v^{i+1}_{2}\dots $ 
	$v^{k}_1 v^{k}_{2}\dots v^{k}_l$ is a shortest path from $v^i_j$ to $v^k_l$. This clearly shows that $v^i_j$ is not mutually maximally distant to $v^k_l$. Similarly, if  $v^i_j\in U_2(C_{n_i})$ and $v^k_l\in U_2(C_{n_k})$, then $P_2=v^i_jv^i_{j-1}\dots v^{i+1}_1v^{i+1}_{n_{i+1}}\dots v^{k}_1$ $v^{k}_{n_k}\dots v^{k}_l$ is a shortest path from $v^i_j$ to $v^k_l$. This shows that $v^i_j$ is not mutually maximally distant to $v^k_l$. Moreover, If $v^i_j\in U_1(C_{n_i})$ and $v^k_l\in U_2(C_{n_k})$, then $R_1=v^i_jv^i_{j+1}\dots v^{i+1}_1v^{i+1}_{2}\dots v^{k}_1,$ $v^{k}_{n_k}\dots v^{k}_l$ is a shortest path from $v^i_j$ to $v^k_l$. This clearly shows that $v^i_j$ is not mutually maximally distant to $v^k_l$. Similarly, if  $v^i_j\in U_2(C_{n_i})$ and $v^k_l\in U_1(C_{n_k})$, then $R_2=v^i_jv^i_{j-1}\dots v^{i+1}_1v^{i+1}_{2}\dots v^{k}_1,$ $v^{k}_{2}\dots v^{k}_l$ is a shortest path from $v^i_j$ to $v^k_l$, which shows that $v^i_j$ is not mutually maximally distant to $v^k_l$.
\end{proof}
For each $i\in \{1,2,\dots, m\}$, Theorem \ref{T1} $(a)$ implies
\begin{equation}\label{e}
A=\{v^{i}_{j}v^{i}_{j+\frac{n_i}{2}}\mid j=2,3,\dots,\frac{n_i}{2},\frac{n_i}{2}+2,\frac{n_i}{2}+3,\dots,n_i \}\subseteq E(\mathcal{C}_{SR}),
\end{equation}
where $j+\frac{n_i}{2}$ are integers modulo $n_{i}$. Similarly, Theorem \ref{T1} $(b)$ implies $v^1_1v^m_{\frac{n_m}{2}+1}\in E(\mathcal{C}_{SR})$. Thus $E(\mathcal{C}_{SR})=A\cup \{v^1_1v^m_{\frac{n_m}{2}+1}\}$.
\begin{Lem}\label{L3}
	Let $\mathcal{C}(C_{n_1},C_{n_2},\dots,C_{n_m})$ ba a chain cycle constructed by even cycles with respect to the vertices $\{v^{i}_{\frac{n_{i}}{2}+1},v^{i+1}_1\mid i=1,2,\dots, m-1\}$ and each $n_i\geq 4$. Then $\alpha(\mathcal{C}_{SR}( C_{n_1},C_{n_2}$ 
	$,\dots,C_{n_m}) =1+\sum_{i=1}^{m}\frac{n_{i}-2}{2}$.
\end{Lem}
\begin{proof}
	We construct a vertex cover of strong resolving graph of $\mathcal{C}(C_{n_1},C_{n_2},\dots,C_{n_m})$ with minimum cardinality. From \eqref{e}, we note that the vertices $\{v^{i}_{j},v^{i}_{j+\frac{n_i}{2}}\mid j=2,3,\dots,\frac{n_i}{2},\frac{n_i}{2}+2,\frac{n_i}{2}+3,\dots,n_i \}$, for each $i\in \{1,2,\dots,m\}$, form $\sum_{i=1}^{m}\frac{n_{i}-2}{2}$ copies of $K_2$. Thus, the $\sum_{i=1}^{m}\frac{n_{i}-2}{2}$ vertices $\{v^{i}_{j}\mid j=2,3,\dots,\frac{n_i}{2},\frac{n_i}{2}+2,\frac{n_i}{2}+3,\dots,n_i \}$, for each $i\in \{1,2,\dots,m\}$, are minimum number of vertices to cover the edges of $A$. Let $S=\{v^{i}_{j}\mid j=2,3,\dots,\frac{n_i}{2},\frac{n_i}{2}+2,\frac{n_i}{2}+3,\dots,n_i \}$. Furthermore, since $v^1_{1}v^m_{\frac{n_m}{2}+1}\in E(\mathcal{C}_{SR})$. Thus, the vertex cover of the strong resolving graph of chain cycle $\mathcal{C}(C_{n_1},C_{n_2},\dots,C_{n_m})$ with minimum cardinality is $S:=S\cup \{v_1^1\}$. Hence $\alpha(\mathcal{C}_{SR}(C_{n_1},C_{n_2},\dots,C_{n_m}))=1+\sum_{i=1}^{m}\frac{n_{i}-2}{2}$.
\end{proof}
\begin{Thm}
	Let $\{C_{n_i}\}_{i=1}^m$ be $m$ disjoint cycles with each $n_i$ is even and $n_i\geq 4$, then $sdim(\mathcal{C}(C_{n_1},$ 
	$C_{n_2},\dots,C_{n_m}))=1+\sum_{i=1}^{m}\frac{n_{i}-2}{2}$.	
\end{Thm}
\begin{proof}
	The proof follows from Lemma \ref{L3} and Theorem \ref{T3}.
\end{proof} 
In the next theorem, we find the mutually maximally distant vertices in chain cycle  $\mathcal{C}(C_{n_1},C_{n_2},$ 
$\dots,C_{n_m})$, with each $n_i$ is odd, with respect to the vertices $\{v^{i}_{\frac{n_{i}+1}{2}+1}, v^{i+1}_{1} \mid i=1,2,\dots,m-1\}$.
\begin{Thm}\label{T2}	
	Let $v^i_j,v^k_l\in V_2$, where $i,k\in \{1,2,\dots,m\}$, in a chain cycle $\mathcal{C}(C_{n_1},C_{n_2},$
	$\dots,C_{n_m})$ constructed by odd cycles with respect to the vertices $\{v^{i}_{\frac{n_{i}+1}{2}+1}, v^{i+1}_{1} \mid i=1,2,\dots,m-1\}$.
	\begin{description}
		\item[$ (a)$] 
		Let $i=k$. Then $v^i_j$MMD$v^i_l$ if and only if $d(v^i_j,v^i_l)=d(C_{n_i})$. 
		\item[$(b)$]
		Let $i\neq k$. 
		
		\begin{description}
			\item[$(1)$] 
			For $i=1$ and $k=m$, then $v^1_j$MMD$v^m_l$ if and only if $j\in\{1,2\}$ and $l\in \{\frac{n_m+1}{2},\frac{n_m+1}{2}+1\}$,
			\item[$(2)$] 
			For $i=1$ and $k\in \{2,3,\dots,m-1\}$, then $v^1_j$MMD$v^k_l$ if and only if $j\in\{1,2\}$ and $l=\frac{n_k+1}{2}$
			\item[$(3)$] 
			For $i\in \{2,3,\dots,m-1\}$ and $k=m$, then $v^i_l$MMD$v^m_j$ if and only if $j=2$ and $l\in\{\frac{n_m+1}{2},\frac{n_m+1}{2}+1\}$,
			\item[$(4)$]
			For $i\in \{2,3,\dots,m-1\}$ and $k\in \{2,3,\dots,m-1\}$, then $v^i_j$MMD$v^k_l$ if and only if $j=2$ and $l=\frac{n_k+1}{2}$.
		\end{description}
	\end{description}
\end{Thm}
\begin{proof}
	(a). Proof is similar to the proof of Theorem \ref{T1} part (a).
	
	(b). Let $i\neq k$. We prove the cases $(1)$ and $(2)$ and the proof of the cases $(3)$ and $(4)$ are similar. 
	
	$(1)$. Suppose that $j\in \{1,2\}$ and $l\in \{\frac{n_m+1}{2},\frac{n_m+1}{2}+1\}$. We shall show that  $v^1_j$MMD$v^m_l$. 
	Note that $N(v_1^1)=\{v_2^1,v^1_{n_1}\}$ and $N(v_2^1)=\{v_3^1,v^1_{1}\}$. Then $P_1:=v^1_1v^1_{n_1}v^1_{n_1-1}\dots v^2_1v^2_{n_2}\dots v^m_1v^m_{n_m}\dots$
	$  v^m_{\frac{n_m+1}{2}+2}v^m_{\frac{n_m+1}{2}+1}$ and $P_2:=v^1_2v^1_{3}v^1_{4}\dots v^2_1v^2_{n_2}\dots v^m_1v^m_{n_m}$
	$\dots  v^m_{\frac{n_m+1}{2}+2}v^m_{\frac{n_m+1}{2}+1}$ are shortest path from $v_1^1$ to $v^m_{\frac{n_m+1}{2}+1}$ and from $v^1_2$ to $v^m_{\frac{n_m+1}{2}+1}$, respectively, of length $\sum_{i=1}^m\frac{n_i-1}{2}=s$. Thus from $P_1$ and $P_2$, we have 
	\begin{center}
		$d(v^1_{n_1},v^m_{\frac{n_m+1}{2}+1})=d(v^1_{3},v^m_{\frac{n_m+1}{2}+1})=s-1$.
	\end{center}  
	That is, $v^1_j$MD$v^m_{\frac{n_m+1}{2}+1}$ for $j\in \{1,2\}$.\\ 
	Again note that $N(v^m_{\frac{n_m+1}{2}+1})=\{v^m_{\frac{n_m+1}{2}+2},v^m_{\frac{n_m+1}{2}}\}$. Then $Q_1:=v^m_{\frac{n_m+1}{2}+1}v^m_{\frac{n_m+1}{2}+2}$
	$\dots v^m_1v^{m-1}_{\frac{n_{m-1}+1}{2}+2}$
	$v^{m-1}_{\frac{n_{m-1}+1}{2}+3}\dots v^2_1v^1_{\frac{n_{1}+1}{2}+2}\dots v^1_{n_1}v^1_1$ and $Q_2:=v^m_{\frac{n_m+1}{2}+1}v^m_{\frac{n_m+1}{2}+2}\dots v^m_1$
	
	$v^{m-1}_{\frac{n_{m-1}+1}{2}+2}v^{m-1}_{\frac{n_{m-1}+1}{2}+3}\dots v^2_1v^1_{\frac{n_{1}+1}{2}}\dots$
	$ v^1_{3}v^1_2$ are shortest path from $v^m_{\frac{n_m+1}{2}+1}$ to $v_1^1$ and from $v^m_{\frac{n_m+1}{2}+1}$ to $v^1_2$  of lenght $s$.  Thus from $Q_1$ and $Q_2$, we have 
	\begin{equation}\label{T4e1} 
	d(v^1_{1},v^m_{\frac{n_m+1}{2}+2})=d(v^1_{2},v^m_{\frac{n_m+1}{2}+2})=s-1.
	\end{equation}  
	Also, $Q_3:=v^m_{\frac{n_m+1}{2}}v^m_{\frac{n_m+1}{2}-1}\dots v^m_1v^{m-1}_{\frac{n_{m-1}+1}{2}+2}v^{m-1}_{\frac{n_{m-1}+1}{2}+3}\dots v^2_1v^1_{\frac{n_{1}+1}{2}+2}\dots v^1_{n_1}v^1_1$ and $Q_4:=v^m_{\frac{n_m+1}{2}}v^m_{\frac{n_m+1}{2}-1}$
	$\dots v^m_1v^{m-1}_{\frac{n_{m-1}+1}{2}+2}v^{m-1}_{\frac{n_{m-1}+1}{2}+3}\dots v^2_1v^1_{\frac{n_{1}+1}{2}}\dots v^1_{3}v^1_2$ are shortest path from $v^m_{\frac{n_m+1}{2}}$ to $v_1^1$ and from $v^m_{\frac{n_m+1}{2}}$ to $v^1_2$  of lenght $s$. Thus from $Q_3$, $Q_4$ and equation \eqref{T4e1}, we have $v^m_{\frac{n_m+1}{2}+1}$MD$v^1_j$ for $j\in \{1,2\}$. Hence $v^1_j$MMD$v^m_{\frac{n_m+1}{2}+1}$ for $j\in \{1,2\}$. Similarly, we can prove that $v^1_j$MMD$v^m_{\frac{n_m+1}{2}}$ for $j\in \{1,2\}$. Summing up, we have $v^1_j$MMD$v^m_{l}$ for $j\in \{1,2\}$ and $l\in \{\frac{n_m+1}{2},\frac{n_m+1}{2}+1\}$.
	
	Conversely, Suppose $v^1_j$MMD$v^m_{l}$ then we show that $j\in \{1,2\}$ and $l\in\{\frac{n_m+1}{2},\frac{n_m+1}{2}+1\}$. On the contrary, we shall prove the following cases:\\
	\textbf{Case $1$}:  $j\notin \{1,2\}$ but $l\in\{\frac{n_m+1}{2},\frac{n_m+1}{2}+1\}$,\\
	\textbf{Case $2$}: $l\notin\{\frac{n_m+1}{2},\frac{n_m+1}{2}+1\}$ but $j\in \{1,2\}$,\\
	\textbf{Case $3$}: $j\notin \{1,2\}$ and $l\notin\{\frac{n_m+1}{2},\frac{n_m+1}{2}+1\}$. 
	
	\textbf{Case $1$}: First suppose that $j\notin \{1,2\}$ and $l\in\{\frac{n_m+1}{2},\frac{n_m+1}{2}+1\}$. Let $j\in \{3,4,\dots,\frac{n_1+1}{2}\}$. Note that $N(v^1_j)=\{v^1_{j-1},v^1_{j+1}\}$. Then $v^1_jv^1_{j+1}v^1_{j+2}\dots v^2_{1}v^2_{n_2}\dots v^m_1v^m_{n_m}$
	$\dots  v^m_{\frac{n_m+1}{2}+2}v^m_{\frac{n_m+1}{2}+1}$ is a shortest path from $v^1_j$ to $v^m_{\frac{n_m+1}{2}+1}$ of length say $r$. But $v^1_{j-1}v^1_jv^1_{j+1}v^1_{j+2}\dots v^2_{1}v^2_{n_2}$
	$\dots v^m_1v^m_{n_m}\dots  v^m_{\frac{n_m+1}{2}+2}v^m_{\frac{n_m+1}{2}+1}$ is a shortest path from $v^1_{j-1}$ to $v^m_{\frac{n_m+1}{2}+1}$ of length $r+1$. That is,
	\begin{center}
		$d(v^1_{j-1},v^m_{\frac{n_m+1}{2}+1})=r+1$.
	\end{center}
	Thus $v^1_j$ and $v^m_{\frac{n_m+1}{2}+1}$ are not MMD.\\ Now let $j\in \{\frac{n_1+1}{2}+2,\frac{n_1+1}{2}+3,\dots,n_1\}$.   
	Note that $N(v^1_j)=\{v^1_{j-1},v^1_{j+1}\}$. Then $v^1_jv^1_{j-1}v^1_{j-2}\dots v^2_{1}v^2_{n_2}$
	$\dots v^m_1v^m_{n_m}\dots  v^m_{\frac{n_m+1}{2}+2}v^m_{\frac{n_m+1}{2}+1}$ is a shortest path from $v^1_j$ to $v^m_{\frac{n_m+1}{2}+1}$ of length say $r'$. But $v^1_{j+1}v^1_jv^1_{j-1}v^1_{j-2}$
	$\dots v^2_{1}v^2_{n_2}\dots v^m_1v^m_{n_m}\dots  v^m_{\frac{n_m+1}{2}+2}v^m_{\frac{n_m+1}{2}+1}$ is a shortest path from $v^1_j$ to $v^m_{\frac{n_m+1}{2}+1}$ of length $r'+1$.
	That is,
	\begin{center}
		$d(v^1_{j+1},v^m_{\frac{n_m+1}{2}+1})=r'+1$.
	\end{center} 
	Thus $v^1_{j}$ and $v^m_{\frac{n_m+1}{2}+1}$ are not MMD for $j\notin\{1,2\}$. Similarly, we can prove that $v^1_{j}$ and $v^m_{\frac{n_m+1}{2}}$ are not MMD for $j\notin\{1,2\}$. Summing up, we have $v^1_{j}$ and $v^m_{l}$ are not MMD for $j\notin\{1,2\}$ and $l\in\{\frac{n_m+1}{2},\frac{n_m+1}{2}+1\}$.
	
	\textbf{Case $2$}: Secondly, suppose that $l\notin\{\frac{n_m+1}{2},\frac{n_m+1}{2}+1\}$ and $j\in\{1,2\}$. Let $l\in\{2,3,\dots \frac{n_m+1}{2}-1\}$. Note that $N(v^m_l)=\{v^m_{l-1},v^m_{l+1}\}$. Then $v^m_{l}v^m_{l-1}\dots v^m_1v^{m-1}_{\frac{n_m+1}{2}+2}\dots v^2_1$
	$v^1_{\frac{n_m+1}{2}+2}\dots v^1_{n_1}v^1_1$ 
	is a shortest path from $v^m_l$ to $v^1_{1}$ of length say $t$. But $v^m_{l+1}v^m_{l}v^m_{l-1}\dots $
	$v^m_1v^{m-1}_{\frac{n_m+1}{2}+2}\dots v^2_1v^1_{\frac{n_m+1}{2}+2}\dots v^1_{n_1}v^1_1$ is a shortest path from $v^m_{l+1}$ to $v^1_{1}$ of length $t+1$.
	That is,
	\begin{center}
		$d(v^1_{1},v^m_{l+1})=t+1$.
	\end{center} 
	Thus $v^m_l$ and $v^1_1$ are not MMD. \\
	Now let $l\in\{\frac{n_m+1}{2}+2,\frac{n_m+1}{2}+3,\dots n_m\}$. Note that $N(v^m_l)=\{v^m_{l-1},v^m_{l+1}\}$. Then $v^m_{l}v^m_{l+1}\dots v^m_1v^{m-1}_{\frac{n_m+1}{2}+2}$
	$\dots v^2_1v^1_{\frac{n_m+1}{2}+2}\dots v^1_{n_1}v^1_1$ 
	is a shortest path from $v^m_l$ to $v^1_{1}$ of length say $t'$. But $v^m_{l-1}v^m_{l}v^m_{l+1}\dots v^m_1v^{m-1}_{\frac{n_m+1}{2}+2}$
	$\dots v^2_1v^1_{\frac{n_m+1}{2}+2}\dots v^1_{n_1}v^1_1$ is a shortest path from $v^m_{l-1}$ to $v^1_{1}$ of length $t'+1$.
	That is,
	\begin{center}
		$d(v^1_{1},v^m_{l-1})=t'+1$.
	\end{center} 
	This implies, $v^m_l$ and $v^1_1$ are not MMD. Hence $v^m_l$ and $v^1_1$ are not MMD for $l\notin\{\frac{n_m+1}{2},\frac{n_m+1}{2}+1\}$. Similarly, we can prove that $v^m_l$ and $v^1_2$ are not MMD for $l\notin\{\frac{n_m+1}{2},\frac{n_m+1}{2}+1\}$. Summing up, we have $v^m_l$ and $v^1_2$ are not MMD for $l\notin\{\frac{n_m+1}{2},\frac{n_m+1}{2}+1\}$ and $j\in\{1,2\}$. 
	
	\textbf{Case $3$}: Thirdly, suppose that $l\notin\{\frac{n_m+1}{2},\frac{n_m+1}{2}+1\}$ and $j\notin\{1,2\}$. Then the proof is straigt forward from the cases $l\notin\{\frac{n_m+1}{2},\frac{n_m+1}{2}+1\}$ or $j\notin\{1,2\}$. This completes the proof of $(1)$.
	
	$(2)$. Suppose that $j\in \{1,2\}$ and $l=\frac{n_k+1}{2}$. We shall show that  $v^1_j$MMD$v^k_l$. 
	Note that $N(v_1^1)=\{v_2^1,v^1_{n_1}\}$ and $N(v_2^1)=\{v_3^1,v^1_{1}\}$. Then $P_3:=v^1_1v^1_{n_1}v^1_{n_1-1}\dots v^2_1v^2_{n_2}\dots v^k_1v^k_{2}\dots$
	$  v^k_{\frac{n_k+1}{2}-1}v^k_{\frac{n_k+1}{2}}$ and $P_4:=v^1_2v^1_{3}v^1_{4}\dots v^2_1v^2_{n_2}\dots v^k_1v^k_{2}\dots  v^k_{\frac{n_k+1}{2}-1}v^k_{\frac{n_k+1}{2}}$ are shortest path from $v_1^1$ to $v^k_{\frac{n_k+1}{2}}$ and from $v^1_2$ to $v^k_{\frac{n_k+1}{2}}$, respectively, of length $\sum_{i=1}^k\frac{n_i-1}{2}=s$. Thus from $P_3$ and $P_4$, we have 
	\begin{center}
		$d(v^1_{n_1},v^k_{\frac{n_k+1}{2}})=d(v^1_{3},v^k_{\frac{n_k+1}{2}+1})=s-1$.
	\end{center}  
	That is, $v^1_j$MD$v^k_{\frac{n_k+1}{2}}$ for $j\in \{1,2\}$ and $l=\frac{n_k+1}{2}$.\\ 
	Again note that $N(v^k_{\frac{n_k+1}{2}})=\{v^k_{\frac{n_k+1}{2}+1},v^k_{\frac{n_k+1}{2}-1}\}$. Then $R_1:=v^k_{\frac{n_k+1}{2}}v^k_{\frac{n_k+1}{2}-1}\dots v^m_1$
	$v^{m-1}_{\frac{n_{m-1}+1}{2}+2}v^{m-1}_{\frac{n_{m-1}+1}{2}+3}$
	$\dots v^2_1v^1_{\frac{n_{1}+1}{2}+2}\dots v^1_{n_1}v^1_1$ and $R_2:=v^k_{\frac{n_k+1}{2}}v^k_{\frac{n_k+1}{2}-1}\dots v^k_1v^{k-1}_{\frac{n_{k-1}+1}{2}+2}$
	
	$v^{k-1}_{\frac{n_{k-1}+1}{2}+3}\dots v^2_1v^1_{\frac{n_{1}+1}{2}}\dots v^1_{3}v^1_2$ are shortest path from $v^k_{\frac{n_k+1}{2}}$ to $v_1^1$ and from $v^k_{\frac{n_k+1}{2}}$ to $v^1_2$  of lenght $s$.  Thus from $R_1$ and $R_2$, we have 
	\begin{equation}\label{T4e2} 
	d(v^1_{1},v^k_{\frac{n_k+1}{2}-1})=d(v^1_{2},v^k_{\frac{n_k+1}{2}-1})=s-1.
	\end{equation}  
	Also, $R_3:=v^k_{\frac{n_k+1}{2}+1}v^k_{\frac{n_k+1}{2}+2}\dots v^k_1v^{k-1}_{\frac{n_{k-1}+1}{2}+2}v^{k-1}_{\frac{n_{k-1}+1}{2}+3}\dots v^2_1v^1_{\frac{n_{1}+1}{2}+2}\dots v^1_{n_1}v^1_1$ and $R_4:=v^k_{\frac{n_k+1}{2}+1}v^k_{\frac{n_k+1}{2}+2}$
	$\dots v^k_1v^{k-1}_{\frac{n_{k-1}+1}{2}+2}v^{k-1}_{\frac{n_{k-1}+1}{2}+3}\dots v^2_1v^1_{\frac{n_{1}+1}{2}}\dots v^1_{3}v^1_2$ are shortest path from $v^k_{\frac{n_k+1}{2}+1}$ to $v_1^1$ and from $v^k_{\frac{n_k+1}{2}+1}$ to $v^1_2$  of lenght $s$. Thus from $R_3$, $R_4$ and equation \eqref{T4e2}, we have $v^k_{\frac{n_k+1}{2}}$MD$v^1_j$ for $j\in \{1,2\}$. Hence $v^1_j$MMD$v^k_{\frac{n_k+1}{2}}$ for $j\in \{1,2\}$. Thus we have $v^1_j$MMD$v^m_{l}$ for $j\in \{1,2\}$ and $l=\frac{n_m+1}{2}$.
	
	Conversely, Suppose $v^1_j$MMD$v^k_{l}$ and we show that $j\in \{1,2\}$ and $l=\frac{n_k+1}{2}$. On the contrary, we shall prove the following cases:\\
	\textbf{Case $1$}: $j\notin \{1,2\}$ but $l=\frac{n_k+1}{2}$,\\
	\textbf{Case $2$}:
	$l\neq\frac{n_k+1}{2}$ but $j\in \{1,2\}$,\\
	 \textbf{Case $3$}: $j\notin \{1,2\}$ and $l\neq\frac{n_k+1}{2}$. 
	
	\textbf{Case $1$}: First suppose that $j\notin \{1,2\}$ and $l=\frac{n_k+1}{2}$. Let $j\in \{3,4,\dots,\frac{n_1+1}{2}\}$. Note that $N(v^1_j)=\{v^1_{j-1},v^1_{j+1}\}$. Then $v^1_jv^1_{j+1}v^1_{j+2}\dots v^2_{1}v^2_{n_2}\dots v^k_1v^k_{2}\dots  v^k_{\frac{n_k+1}{2}-1}v^k_{\frac{n_k+1}{2}}$ is a shortest path from $v^1_j$ to $v^k_{\frac{n_k+1}{2}}$ of length say $x$. But $v^1_{j-1}v^1_jv^1_{j+1}v^1_{j+2}\dots v^2_{1}v^2_{n_2}\dots v^k_1v^k_{2}\dots  v^k_{\frac{n_k+1}{2}-1}$
	
	$v^k{\frac{n_k+1}{2}}$ is a shortest path from $v^1_{j-1}$ to $v^k_{\frac{n_k+1}{2}}$ of length $x+1$. That is,
	\begin{center}
		$d(v^1_{j-1},v^k_{\frac{n_k+1}{2}})=x+1$.
	\end{center}
	Thus $v^1_j$ and $v^k_{\frac{n_k+1}{2}}$ are not MMD.\\ Now let $j\in \{\frac{n_1+1}{2}+2,\frac{n_1+1}{2}+3\dots,n_1\}$.   
	Note that $N(v^1_j)=\{v^1_{j-1},v^1_{j+1}\}$. Then $v^1_jv^1_{j-1}v^1_{j-2}\dots v^2_{1}v^2_{n_2}$
	$\dots v^k_1v^k_{2}\dots  v^k_{\frac{n_k+1}{2}-1}v^k_{\frac{n_k+1}{2}}$ is a shortest path from $v^1_j$ to $v^k_{\frac{n_k+1}{2}}$ of length say $x'$. But $v^1_{j+1}v^1_jv^1_{j-1}v^1_{j-2}$
	$\dots v^2_{1}v^2_{n_2}\dots v^k_1v^k_{2}\dots  v^k_{\frac{n_k+1}{2}-1}v^k_{\frac{n_k+1}{2}}$ is a shortest path from $v^1_j$ to $v^k_{\frac{n_k+1}{2}}$ of length $x'+1$.
	That is,
	\begin{center}
		$d(v^1_{j+1},v^k_{\frac{n_k+1}{2}})=x'+1$.
	\end{center} 
	Thus $v^1_{j}$ and $v^k_{\frac{n_k+1}{2}}$ are not MMD for $j\notin\{1,2\}$. Summing up, we have $v^1_{j}$ and $v^m_{l}$ are not MMD for $j\notin\{1,2\}$ and $l=\frac{n_k+1}{2}$.
	
	\textbf{Case $2$}: Secondly, suppose that $l\neq\frac{n_k+1}{2}$ and $j\in\{1,2\}$. Let $l\in\{2,3,\dots \frac{n_k+1}{2}-1\}$. Note that $N(v^k_l)=\{v^k_{l-1},v^k_{l+1}\}$. Then $v^k_{l}v^k_{l-1}\dots v^k_1v^{k-1}_{\frac{n_k+1}{2}+2}\dots v^2_1v^1_{\frac{n_k+1}{2}+2}\dots v^1_{n_1}v^1_1$ 
	is a shortest path from $v^m_l$ to $v^1_{1}$ of length say $y$. But $v^k_{l+1}v^k_{l}v^k_{l-1}\dots v^k_1v^{k-1}_{\frac{n_k+1}{2}+2}\dots v^2_1v^1_{\frac{n_k+1}{2}+2}\dots v^1_{n_1}v^1_1$ is a shortest path from $v^k_{l+1}$ to $v^1_{1}$ of length $y+1$.
	That is,
	\begin{center}
		$d(v^1_{1},v^k_{l+1})=y+1$.
	\end{center} 
	Thus $v^k_l$ and $v^1_1$ are not MMD. \\
	Now let $l\in\{\frac{n_m+1}{2}+2,\frac{n_m+1}{2}+3,\dots n_k\}$. Note that $N(v^k_l)=\{v^k_{l-1},v^k_{l+1}\}$. Then $v^k_{l}v^k_{l+1}\dots v^k_1v^{k-1}_{\frac{n_k+1}{2}+2}$
	$\dots v^2_1v^1_{\frac{n_k+1}{2}+2}\dots v^1_{n_1}v^1_1$ 
	is a shortest path from $v^k_l$ to $v^1_{1}$ of length say $y'$. But $v^k_{l-1}v^k_{l}v^k_{l+1}\dots v^k_1v^{k-1}_{\frac{n_k+1}{2}+2}$
	$\dots v^2_1v^1_{\frac{n_k+1}{2}+2}\dots v^1_{n_1}v^1_1$ is a shortest path from $v^k_{l-1}$ to $v^1_{1}$ of length $y'+1$.
	That is,
	\begin{center}
		$d(v^1_{1},v^k_{l-1})=y'+1$.
	\end{center} 
	This implies, $v^k_l$ and $v^1_1$ are not MMD. Hence $v^k_l$ and $v^1_1$ are not MMD for $l\neq\frac{n_k+1}{2}$.  Summing up, we have $v^m_l$ and $v^1_j$ are not MMD for $l\neq\frac{n_k+1}{2}$ and $j\in\{1,2\}$. 
	
	\textbf{Case $3$}: Thirdly, let $j\notin \{1,2\}$ and $l\neq\frac{n_k+1}{2}$. Then the proof is straigt forward from the cases $j\notin \{1,2\}$ or $l\neq\frac{n_k+1}{2}$. This completes the proof of $(2)$.
\end{proof}
From Theorem \ref{T2} (a),  for all $i\in \{2,3,\dots, m-1\}$, we have
\begin{eqnarray}
&A_1=\{v^1_jv^1_{j+\frac{n_1-1}{2}}\mid j=1,2,\dots,\frac{n_1+1}{2},\frac{n_1+1}{2}+2,\dots n_1\}\subseteq E(C_{SR}),\label{T2eq1}\\
&A_2=\{v^m_jv^m_{j+\frac{n_m-1}{2}}\mid j=2,3,\dots, n_m\}\subseteq E(C_{SR}),\\
&A_3=\{v^i_jv^i_{j+\frac{n_i-1}{2}}\mid j=2,3,\dots,\frac{n_i+1}{2},\frac{n_i+1}{2}+2,\dots n_i\}\subseteq E(C_{SR}).
\end{eqnarray}
From Theorem \ref{T2} (b), for all $i,k\in \{2,3,\dots,m-1\}$, we have 
\begin{eqnarray}
&B_1=\{v^1_jv^m_{l}\mid j=1,2. \  and\ l=\frac{n_m+1}{2},\frac{n_m+1}{2}+1.\}\subseteq E(C_{SR}),\\
&B_2=\{v^1_jv^k_{\frac{n_k+1}{2}}\mid j=1,2.\}\subseteq E(C_{SR}),\\
&B_3=\{v^i_2v^m_{l}\mid l=\frac{n_m+1}{2},\frac{n_m+1}{2}+1.\}\subseteq E(C_{SR}),\\
&B_4=\{v^i_2v^k_{\frac{n_k+1}{2}}\}\subseteq E(C_{SR})\label{T2eq2}.
\end{eqnarray}
Thus from \eqref{T2eq1}$\sim$\eqref{T2eq2}, we have $A_1\cup A_2\cup A_3\cup B_1\cup B_2\cup B_3\cup B_4=E(C_{SR})$.

Note that the set of edges $A_1$ form a path $P_{n_1-1}$ with initial vertex $v^1_1$ and final vertex $v^1_2$, that is, 
\begin{equation*}
P_{n_1-1}:=v^1_1v^1_{1+\lfloor\frac{n_1}{2}\rfloor}v^1_{1+2\lfloor\frac{n_1}{2}\rfloor}\dots v^1_{1+(n_1-2)\lfloor\frac{n_1}{2}\rfloor},
\end{equation*} 
where $1+(n_1-2)\lfloor\frac{n_1}{2}\rfloor\equiv 2\ ( mod\ n_1)$.

Similarly, the set of edges $A_2$ form a path $P_{n_m-1}$ with initial vertex $v^m_{\frac{n_m+1}{2}}$ and final vertex $v^m_{\frac{n_m+1}{2}+1}$, that is,
\begin{equation*}
P_{n_m-1}:=v^m_{\frac{n_m+1}{2}}v^m_{\frac{n_m+1}{2}+\lfloor\frac{n_m}{2}\rfloor}v^m_{\frac{n_m+1}{2}+2\lfloor\frac{n_m}{2}\rfloor}\dots v^m_{\frac{n_m+1}{2}+(n_m-2)\lfloor\frac{n_m}{2}\rfloor},
\end{equation*}
where $\frac{n_m+1}{2}+(n_m-2)\lfloor\frac{n_m}{2}\rfloor\equiv \frac{n_m+1}{2}+1 \ ( mod \ n_m)$.

Also, the set of edges $A_3$ form $m-2$ paths $P_{n_i-2}$, $i\in \{2,3,\dots, m-1\}$, with initial vertex $v^i_{2}$ and final vertex $v^i_{\frac{n_i+1}{2}}$, that is, 
\begin{equation*}
P_{n_i-2}:=v^i_2v^i_{2+\lceil\frac{n_i}{2}\rceil}v^i_{2+2\lceil\frac{n_i}{2}\rceil}\dots v^i_{2+(n_i-3)\lceil\frac{n_i}{2}\rceil},
\end{equation*}
where $2+(n_i-3)\lceil\frac{n_i}{2}\rceil \equiv \frac{n_i+1}{2} \ ( mod \ n_i)$. 
\begin{Lem}\label{L4}
	Let $\mathcal{C}(C_{n_1},C_{n_2},\dots,C_{n_m})$ ba a chain cycle constructed by odd cycles with respect to the vertices $\{v^{i}_{\frac{n_{i}+1}{2}+1}, v^{i+1}_{1} \mid i=1,2,\dots,m-1\}$ and each $n_i\geq 5$. Then $\alpha(\mathcal{C}_{SR}( C_{n_1},C_{n_2}$ 
	$,\dots,C_{n_m}) =m-1+\lfloor\frac{n_1}{2}\rfloor+\lfloor\frac{n_m}{2}\rfloor+\sum_{i=2}^{m-1}\lfloor\frac{n_i-2}{2}\rfloor$.
\end{Lem}
\begin{proof}
	We construct a vertex cover of  $\mathcal{C}_{SR}( C_{n_1},C_{n_2},\dots,C_{n_m})$ with minimum cardinality. Note that the $m$ vertices $v^1_1$ and $v^i_2$, $i\in \{2,3,\dots,m-1\}$, are nonadjacent vertices in $\mathcal{C}_{SR}( C_{n_1},C_{n_2},\dots,C_{n_m})$ and cover all edges of sets $B_1$, $B_2$, $B_3$ and $B_4$. Let $S=\{v^1_1,v^i_2\mid i=2,3,\dots,m-1\}$. In order to cover the edges of path $P_{n_1-1}$ we need $\frac{n_1-1}{2}$ vertices. Since $v^1_1,v^1_2\in S$, therefore we must take $\frac{n_1-1}{2}-1$ more vertices of $P_{n_1-1}$. Thus we take the vertices $v^1_{1+2\lfloor\frac{n_1}{2}\rfloor},v^1_{1+4\lfloor\frac{n_1}{2}\rfloor},\dots,v^1_{1+(n_1-3)\lfloor\frac{n_1}{2}\rfloor}$ in $S$. That is, we augment the set $S$ by taking $S:=S\cup\{v^1,v^1_{1+2\lfloor\frac{n_1}{2}\rfloor},v^1_{1+4\lfloor\frac{n_1}{2}\rfloor},\dots,v^1_{1+(n_1-3)\lfloor\frac{n_1}{2}\rfloor}\}$. Similarly, to cover the edges of the path $P_{n_m-1}$, we must take $\frac{n_m-1}{2}$ vertices of $P_{n_m-1}$. Thus we take the vertices $v^m_{\frac{n_m+1}{2}},v^m_{\frac{n_m+1}{2}+2\lfloor\frac{n_m}{2}\rfloor},v^m_{\frac{n_m+1}{2}+4\lfloor\frac{n_m}{2}\rfloor},\dots,v^m_{\frac{n_m+1}{2}+(n_m-3)\lfloor\frac{n_m}{2}\rfloor}$ in $S$. That is, we again augment the set $S$ by taking $S:=S\cup\{v^m_{\frac{n_m+1}{2}},v^m_{\frac{n_m+1}{2}+2\lfloor\frac{n_m}{2}\rfloor},v^m_{\frac{n_m+1}{2}+4\lfloor\frac{n_m}{2}\rfloor},\dots,$
	$v^m_{\frac{n_m+1}{2}+(n_m-3)\lfloor\frac{n_m}{2}\rfloor}\}$. Finally, to cover the edges of each path $P_{n_i-2}$ we must take $\lceil\frac{n_i-2}{2}\rceil$ vertices of $P_{n_i-1}$, $i\in \{2,3,\dots,m-1\}$. Since $v^i_2\in S$, therefore we take $\lfloor\frac{n_i-2}{2}\rfloor$ more vertices of $P_{n_i-1}$, $i\in\{2,3,\dots,m-1\}$. Thus we take the vertices $v^i_{2+2\lceil\frac{n_i}{2}\rceil},v^i_{2+4\lceil\frac{n_i}{2}\rceil},\dots,v^i_{2+(n_i-3)\lceil\frac{n_i}{2}\rceil}$, $i\in\{2,3,\dots,m-2\}$ in $S$. That is, we again augment the set $S$ by taking $S:=S\cup\{v^i_{2+2\lceil\frac{n_i}{2}\rceil},v^i_{2+4\lceil\frac{n_i}{2}\rceil},\dots,v^i_{2+(n_i-3)\lceil\frac{n_i}{2}\rceil}\mid i=2,3,\dots,m-2\}$. Thus $S$ is the vertex cover of the strong resolving graph of the chain cycle $\mathcal{C}(C_{n_1},$ 
	$C_{n_2},\dots,C_{n_m})$ with minimum cardinality $|S|=m-1+\lfloor\frac{n_1}{2}\rfloor+\lfloor\frac{n_m}{2}\rfloor+\sum_{i=2}^{m-1}\lfloor\frac{n_i-2}{2}\rfloor$. 
\end{proof}
\begin{Thm}
	Let $\{C_{n_i}\}_{i=1}^m$ be $m$ disjoint cycles with each $n_i$ is odd and $n_i\geq 5$, then $sdim(\mathcal{C}(C_{n_1},$ 
	$C_{n_2},\dots,C_{n_m}))=m-1+\lfloor\frac{n_1}{2}\rfloor+\lfloor\frac{n_m}{2}\rfloor+\sum_{i=2}^{m-1}\lfloor\frac{n_i-2}{2}\rfloor$.	
\end{Thm}
\begin{proof}
	The proof follows from Lemma \ref{L4} and Theorem \ref{T3}.
\end{proof}



\section*{\textbf{Acknowledgement}}We would like to thank the editor and the referees. We would also like to thank the Higher Education Commission of Pakistan to support this research under grant No. 20-3067/NRPU /R$\&$D/HEC/12.


\begin{thebibliography}{99}
	\bibitem{A}
	S. Akhter, R. Farooq, Metric dimension of fullerene graphs, \emph{Electronic Journal of Graph Theory and Applications}, \textbf{7(1)} (2019), 87--99.
\bibitem{C1} 
G. Chartrand, E. Salehi, P. Zhang, The partition dimension of a graph, \emph{Aequationes Math.}, \textbf{59} (2000),  45--54.

\bibitem{C2} 
G. Chartrand, L. Eroh, M. A. Johnson, O. R. Oellermann, Resolvability in graphs and the metric dimension of a graph, \emph{Discrete Appl. Math.}, \textbf{105} (2000), 99--113.

\bibitem{H1} 
F. Harary, R. A. Melter, On the metric dimension of a graph, \emph{Ars Combin.}, \textbf{2} (1976), 191--195.

\bibitem{I1} 
H. Iswadi, E. T. Baskoro, A. N. M. Salman, R. Simanjuntak, The metric dimension of amalgamation of cycles, \emph{Far East J. Math. Sci.}, \textbf{41(1)} (2010), 19--31.


\bibitem{J1} 
J. A. Rodr{\'\i}guez-Vel{\'a}zquez, I. G. Yero, M. Lemanska, On the partition dimension of trees, \emph{Discrete Appl. Math.}, \textbf{166} (2014), 204--209.	

\bibitem{K1}
D. Kuziak, Strong resolvability in product graphs, Doctoral Thesis, Universitat Rovira i Virgili: The public university of Tarragona, Spain, 2014.	

\bibitem{K2} 
D. Kuziak, J. A. Rodr{\'\i}guez-Vel{\'a}zquez, I. G.  Yero, Ismael G, Computing the metric dimension of a graph from primary subgraphs, \emph{Discuss. Math. Graph Theory}, \textbf{37} (2017), 273--293.

\bibitem{K3}
D. Kuziak, I. G. Yero, J. A.  Rodr{\'\i}guez-Vel{\'a}zquez,  Closed formulae for the strong metric dimension of lexicographic product graphs, \emph{Discuss. Math. Graph Theory}, \textbf{36} (2016), 1051--1064.

\bibitem{M1} 
T. Mansour, M. Schork , The PI index of bridge and chain graphs, \emph{Match}, \textbf{61} (2009), 723.

\bibitem{M2}
N. Mehreen, R. Farooq, S. Akhter, On partition dimension of fullerene graphs, \emph{AIMS Mathematics}, \textbf{3} (2018), 343--352.

\bibitem{N2} 
Nilanjan De, Hyper Zagreb Index of Bridge and Chain Graphs, \emph{arXiv: 1703.08325}, 2017.

\bibitem{O1}
O. R. Oellermann, J. Peters-Fransen, The strong metric dimension
of graphs and digraphs, \emph{Discrete Appl. Math.}, \textbf{155} (2007), 356--364.

\bibitem{R1}
J.A. Rodr{\'\i}guez-Vel{\'a}zquez, I. G. Yero, D. Kuziak, O. R. Oellermann, On the strong metric dimension of Cartesian and direct products of graphs, \emph{Discrete Math.}, \textbf{335} (2014), 8--19.

\bibitem{R2}
J. A. Rodr\'{i}guez-Vel\'{a}zquez, I. G. Yero, H. Fernau, On the partition dimension of unicyclic graphs, \emph{Bull. Math. Soc. Sci. Math. Roumanie}, \textbf{57(4)} (2014), 381--391.


\bibitem{S1}
A. Seb{\H{o}}, E. Tannier, On metric generators of graphs, \emph{Math. Oper. Res.} \textbf{29} (2004), 383--393.

\bibitem{T1}
I. Tomescu, Discrepancies between metric dimension and partition dimension of a connected graph, \emph{Discrete Math.}, \textbf{308} (2008), 5026--5031.

\bibitem{T2}
I. Tomescu, I. Javaid, Slamin, On the partition dimension and connected partition dimension of wheels, \emph{Ars Combin.}, \textbf{84} (2007), 311--317.

\end{thebibliography}
\end{document}